\documentclass[a4paper,11pt]{amsart}

\usepackage{amsmath,amssymb,amsthm}

\theoremstyle{definition}

\newtheorem{thm}{Theorem}
\newtheorem{conj}{Conjecture}
\newtheorem{prop}{Proposition}
\newtheorem{lem}{Lemma}

\newtheorem{de}{Definition}
\newtheorem{ex}{Example}
\newtheorem{rem}{Remark}
\newtheorem{obs}{Observation}
\newtheorem{ques}{Question}

\begin{document}
\title[The number of linear factors of supersingular polynomials]{The number of linear factors of supersingular polynomials \\ and sporadic simple groups}
\author{Tomoaki Nakaya } 
\address{Graduate School of Mathematics, Kyushu University, 744, Motooka, Nishi-ku, Fukuoka, 819-0395, Japan}
\email{t-nakaya@math.kyushu-u.ac.jp}
\date{}
\maketitle

\begin{abstract}
The set of prime numbers $p$ such that the supersingular $j$-invariants in characteristic $p$ are all contained in the prime field is finite.
And it is well known that this set of primes coincides with the set of prime divisors of the order of the Monster simple group.
In this paper, we will present analogous coincidences of supersingular invariants in level 2 and 3 and the orders of the Baby monster group and the Fischer's group. 
The proof uses a connection between the number of supersingular invariants and class numbers of imaginary quadratic fields.
\end{abstract}

\section{Introduction}
An elliptic curve $E$ over a field $K$ of characteristic $p> 0$ is called {\it supersingular} if it has no $p$-torsion over $\overline{K}$.
This condition depends only on the $j$-invariant of $E$, and it is known that there are only finitely many supersingular $j$-invariants, all being contained in $\mathbb{F}_{p^2}$.
We define the supersingular polynomial $ss_{p}(X)$ as the monic polynomial whose roots are exactly all the supersingular $j$-invariants:
\begin{align}
ss_{p}(X) = \prod_{\substack{E/\overline{\mathbb{F}}_{p} \\[1pt] E\,:\, \text{supersingular}}}  (X - j(E)). 
\end{align}
Because the set of supersingular $j$-invariants  in characteristic $p$ is stable under the conjugation over $\mathbb{F}_{p}$, we have $ss_{p}(X) \in \mathbb{F}_{p}[X]$.
For $p=2$ and $3$, we have $ss_{2}(X) = X \pmod{2}$ and $ss_{3}(X)=X \pmod{3}$ (see \cite[p.201]{deuring1941die}). 

For any prime $p \geq 5$, we define the numbers $\nu, \delta, \varepsilon \in \{0,1 \}$,
which will be used throughout this paper, by
\begin{equation}
 \nu = \frac{1}{2}\left( 1- \left( \frac{-2}{p}\right)  \right) ,\;   \delta = \frac{1}{2}\left( 1- \left( \frac{-3}{p}\right)  \right) ,\; \varepsilon = \frac{1}{2}\left( 1- \left( \frac{-1}{p}\right)  \right) ,  \label{eq:nudeleps}
\end{equation}
where $(\tfrac{\cdot }{p})$ is the Legendre symbol.

The following theorem is known (essentially due to \cite{deuring1941die}, see \cite[Proposition 5]{kaneko1998supersingular}).
\begin{thm}\label{thm:DEU}
Let $p \geq 5$ be a prime and $m = [p/12]$. Then
\begin{equation}
ss_p (X) = X^{m+\delta}(X-1728)^{\varepsilon} {}_{2}F_{1}\left( -m ,\frac{5}{12}-\frac{2\delta-3\varepsilon}{6};1;\frac{1728}{X} \right) \pmod{p} , \label{eq:hrepofssp}
\end{equation}
where ${}_{2}F_{1}(\alpha,\beta;\gamma;x)$ is the Gauss hypergeometric series
\[ {}_{2} F_{1}(\alpha , \beta  ; \gamma ; x ) = \sum_{n=0}^\infty \frac{ (\alpha )_{n} (\beta )_{n} }{ (\gamma )_{n}} \frac{x^n}{n!} \quad (|x|<1 ). \]
Here, $(\alpha )_{0}=1$ and $(\alpha )_{n}=\alpha (\alpha +1) \cdots (\alpha +n-1) \,\,\, (n\geq 1)$.
\end{thm}
We note that the series ${}_{2} F_{1}(\alpha , \beta  ; \gamma ; x ) $ becomes a polynomial when $\alpha $ or $\beta $ is a negative integer and $\gamma $ is non-negative integer.
We give a proof of this theorem in Section \ref{Preliminaries} for convenience of the reader.

From this hypergeometric expression of $ss_{p} (X)$ we easily see that
\[ \deg ss_{p} (X) =m +\delta +\varepsilon = \frac{p-1}{12} +\frac{1}{4}\left( 1- \left( \frac{-1}{p}\right)  \right) +\frac{1}{3}\left( 1- \left( \frac{-3}{p}\right)  \right). \]
(See also \cite{eichler1938uber, igusa1958class} and \cite[Ch.V \S 4]{silverman2009arithmetic}.)
The polynomial $ss_{p}(X)$ factors into linear and quadratic polynomials in $\mathbb{F}_{p}[X]$ by the result of Deuring \cite{deuring1941die} that all supersingular $j$-invariants lie in $\mathbb{F}_{p^2}$.
Let $h(\sqrt{-d})$ denote the class number of the imaginary quadratic field $\mathbb{Q}(\sqrt{-d})$. 
The following theorem is essentially due to Deuring~(\cite[\S 10]{deuring1941die} and \cite[eq.(9)]{deuring1944die}), but is expressed in somewhat different form (see also \cite[Lemma 2.6]{pizer1978note}, \cite[p.97]{brillhart2004class}).
\begin{thm}\label{thm:NofLFofssp}
If $p\geq 5$, the number $L(p)$ of supersingular $j$-invariants that lie in $\mathbb{F}_{p}$ i.e. the number of linear factors of $ss_{p}(X)$ is 
\begin{align*}
L(p) &= \frac{1}{4} \left\{ 2+ \left( 1 - \left( \frac{-1}{p}\right) \right) \left( 2 + \left( \frac{-2}{p}\right) \right) \right\} h(\sqrt{-p}) \\
        &= \begin{cases}
              \tfrac{1}{2} h(\sqrt{-p}) & \text{if $p \equiv 1 \pmod{4}$,}\\
              2 h(\sqrt{-p})  & \text{if $p \equiv 3 \pmod{8}$,}\\
              h(\sqrt{-p})  & \text{if $p \equiv 7 \pmod{8}$.}
              \end{cases}
\end{align*}
\end{thm}

Let $\mathbb{M}$ be the Monster group. It is the largest sporadic finite simple group and has order
\begin{align*}
\# \mathbb{M} &=808017424794512875886459904961710757005754368000000000 \\
&=2^{46}\cdot 3^{20}\cdot 5^9\cdot 7^6\cdot 11^2\cdot 13^3\cdot 17\cdot 19\cdot 23\cdot 29\cdot 31\cdot 41\cdot 47\cdot 59\cdot 71.
\end{align*}
In 1975, A. Ogg noticed that the prime divisors $p$ of $\# \mathbb{M}$ agree with those $p$ such that all characteristic $p$ supersingular $j$-invariants that lie in $\mathbb{F}_{p}$ (see \cite[p.7]{ogg1975automorphismes}, \cite{duncan2016jack}).
Ogg offered a bottle of Jack Daniels for an explanation of this coincidence; therefore this is called ``The Jack Daniels Problem''.
By the definition of the polynomial $ss_{p}(X)$ and the number $L(p)$, Ogg's observation is paraphrased as the following theorem.
\begin{thm}\label{thm:NofLFofLevel1ssp}
For a prime number $p$,
\[ \deg ss_{p}(X) = L(p) \iff p\mid \# \mathbb{M}. \]
\end{thm}
This theorem can be proved by using class number estimate, but it does not explain why such a relationship exists.

We consider analogues of these theorems in the case of higher levels or other sporadic groups.
The supersingular polynomial $ss_{p}^{(N*)}(X)$ for the Fricke group $\Gamma_{0}^{*}(N) \;(N=2,3)$ was defined by Koike~\cite{koike2009supersingular} and Sakai~\cite{sakai2011atkin}. It is derived from the invariant differential for $\Gamma_{0}^{*}(N) \;(N=2,3)$ and has hypergeometric series representation.

\begin{de}\label{def:sakaissp}
For primes $p \geq 5$, let $m_{2}=[p/8]$ and $m_{3}=[p/6]$, and set
\begin{align*}
ss^{(2*)}_{p} (X) &= X^{m_{2}+\varepsilon} (X-256)^{\nu } {}_{2}F_{1}\left( -m_{2} ,\frac{3}{8}-\frac{\varepsilon-2\nu}{4};1;\frac{256}{X} \right) \pmod{p}, \\
ss^{(3*)}_{p} (X) &= X^{m_{3}+\delta} (X-108)^{\delta } {}_{2}F_{1}\left( -m_{3} ,\frac{1}{3}+\frac{\delta}{3};1;\frac{108}{X} \right) \pmod{p} .
\end{align*}
\end{de}
For $p \in \{2,3\}$ and $N \in \{2,3\}$, we have $ss_{p}^{(N*)}(X) = X \pmod{p}$.
The degrees of these polynomials are given as follows.
\begin{align*}
\deg ss^{(2*)}_{p} (X) &= m_{2} +\varepsilon +\nu = \frac{p-1}{8} +\frac{3}{8}\left( 1- \left( \frac{-1}{p}\right)  \right) +\frac{1}{4}\left( 1- \left( \frac{-2}{p}\right)  \right), \\[3pt]
\deg ss^{(3*)}_{p} (X) &= m_{3} + 2\delta = \frac{p-1}{6} +\frac{2}{3}\left( 1- \left( \frac{-3}{p}\right)  \right) .
\end{align*}
We shall derive the number of linear factors of these polynomials explicitly as a linear combination of class numbers of the imaginary quadratic fields:
\begin{thm}\label{thm:NofLF}
If $p\geq 5$ is a prime, then the number of linear factors $L^{(N*)}(p)$ of $ss_{p}^{(N*)}(X) \pmod{p}$ for $N=2 \text{ and } 3$ are
\begin{align*}
L^{(2*)}(p) &= \frac{1}{8} \left\{ 2+ \left( 1 - \left( \frac{-1}{p}\right) \right) \left( 4 + \left( \frac{-2}{p}\right) \right) \right\} h(\sqrt{-p}) +\frac{1}{4} h(\sqrt{-2p}) ,\\[3pt]
L^{(3*)}(p) &= \delta  L(p) +\frac{1}{8} \left\{ 2+ \left( 1 + \left( \frac{-1}{p}\right) \right) \left( 2 + \left( \frac{-2}{p}\right) \right) \right\} h(\sqrt{-3 p}). \\
\end{align*}
\end{thm}
An exact analogue of Theorem \ref{thm:NofLFofLevel1ssp} for the sporadic groups $\mathbb{B}$ and $Fi^{\prime}_{24}$ holds true.
Where $\mathbb{B}$ be the Baby monster group and $Fi^{\prime}_{24}$ be the largest of Fischer's groups. The orders of these groups are given by
\begin{align*}
\# \mathbb{B} &= 4154781481226426191177580544000000 \\
&=2^{41}\cdot 3^{13}\cdot 5^{6} \cdot 7^{2} \cdot 11\cdot 13 \cdot 17 \cdot 19 \cdot 23 \cdot 31 \cdot 47 , \\
\# Fi^{\prime}_{24} &= 1255205709190661721292800 \\
&= 2^{21}\cdot 3^{16}\cdot 5^{2}\cdot 7^{3}\cdot 11\cdot 13\cdot 17\cdot 23\cdot 29.
\end{align*}

\begin{thm}\label{thm:relofspgp}
For a prime number $p$,
\begin{align*}
\deg ss^{(2*)}_{p}(X) = L^{(2*)}(p) &{}\iff p\mid \# \mathbb{B} ,\\
\deg ss^{(3*)}_{p}(X) = L^{(3*)}(p) &{}\iff p\mid \# Fi^{\prime}_{24}.
\end{align*}
\end{thm}

We briefly describe the contents of this paper. 
In Section \ref{Preliminaries}, we will introduce the results on the factorization of the Legendre polynomials by Brillhart and Morton, 
and the results on the supersingular polynomials for congruence subgroups of low levels by Tsutsumi.
Together with Definition \ref{def:sakaissp} by Koike and Sakai, these are the essential tools to prove our main theorems.

In Section \ref{Proof of main theorems}, we prove Theorem \ref{thm:NofLF} and \ref{thm:relofspgp} by combining the results in the previous section.
Although the main theme of this paper is these theorems, we also provide various conjectures in Section \ref{Conjectures and Observations}. 
For instance, when the case of levels $N=5 \text{ and } 7$, the sporadic groups Harada-Norton group and Held group appear in the conjecture of the analogue of Theorem \ref{thm:relofspgp} respectively. Besides, more interestingly, it is expected that so-called Ap\'{e}ry-like numbers appear in the coefficients of the squares of these supersingular polynomials.
Finally, we close this paper with an observations of a curious ``duality'' for the primes and the levels of supersingular polynomials.

Quite recently, the author learned from Prof. Ken Ono (via Prof. Kaneko) about the paper \cite{aricheta2018supersingular} by Aricheta.
There a result on the number of supersingular points on some modular curves \textit{not} defined over $\mathbb{F}_{p}$ is presented, and this is very closely related to our Theorem \ref{thm:relofspgp}.


\section{Preliminaries}\label{Preliminaries}
We first prove the hypergeometric expression of the supersingular polynomial $ss_{p}(X)$.
\begin{proof}[Proof of Theorem \ref{thm:DEU}]
We define the monic polynomial $U_{n}^{\varepsilon }(X)$ of degree $n\geq 0$ by
\begin{align*}
X^{n}\, {}_{2}F_{1}\left( \tfrac{1}{12} ,\tfrac{5}{12};1;\tfrac{1728}{X} \right) &= U_{n}^{0}(X) + O \left( \tfrac{1}{X} \right)  , \\
X^{n-1} (X-1728)\, {}_{2}F_{1}\left( \tfrac{7}{12} ,\tfrac{11}{12};1;\tfrac{1728}{X} \right) &= U_{n}^{1}(X) + O \left( \tfrac{1}{X} \right) .
\end{align*}
By \cite[Proposition 5]{kaneko1998supersingular}, we have $ss_{p}(X) = U_{m +\delta +\varepsilon}^{\varepsilon }(X) \pmod{p}$.
Since $p-1=12m +4\delta +6\varepsilon$, the first two parameters of the hypergeometric series in (\ref{eq:hrepofssp}) reduce modulo $p$ to
\begin{align*}
\left( -m ,\frac{5}{12}-\frac{2\delta-3\varepsilon}{6} \right) \equiv \begin{cases}
(\tfrac{1}{12},\tfrac{5}{12}) \pmod{p} & \text{if $p \equiv 1 \pmod{12}$,} \\
(\tfrac{5}{12},\tfrac{1}{12}) \pmod{p} & \text{if $p \equiv 5 \pmod{12}$,}  \\
(\tfrac{7}{12},\tfrac{11}{12}) \pmod{p} & \text{if $p \equiv 7 \pmod{12}$,}  \\
(\tfrac{11}{12},\tfrac{7}{12}) \pmod{p} & \text{if $p \equiv 11 \pmod{12}$.} 
\end{cases}
\end{align*}
Since ${}_{2}F_{1}(a,b;c;x) = {}_{2}F_{1}(b,a;c;x) $, we see that $U_{m +\delta +\varepsilon}^{\varepsilon }(X)$ is congruent to the left-hand side of  (\ref{eq:hrepofssp}) modulo $p$.
\end{proof}
Precisely, Deuring proved in \cite{deuring1941die} that the polynomial $H_{p-1}(X)$ of \cite[Proposition 5]{kaneko1998supersingular} is equal to $ss_{p}(X)$.


\subsection{Factorization of the Legendre polynomials}
Using the theory of elliptic curves, Brillhart and Morton determined the number of linear factors of the following polynomials $W_{m}(x) \pmod{p}$:
\begin{equation}
W_{m}(X) := \sum_{r=0}^{m} \binom{m}{r}^2 X^{r} = {}_{2}F_{1} (-m, -m;1;X). 
\end{equation}
It is well known that the Hasse invariant for elliptic curves in Legendre form
\[  E_{\lambda } :  y^2 = x(x-1)(x-\lambda ) \]
is $W_{(p-1)/2}(\lambda )$ and the elliptic curve $E_{\lambda }$ is supersingular if and only if $W_{(p-1)/2}(\lambda ) \equiv 0 \pmod{p}$. 
\begin{thm}[Brillhart, Morton \cite{brillhart2004class}]\label{thm:numoflinw}
 Let $p \geq 5$ be a prime and $N_{1}(p,m)$ be the number of linear factors of $W_{m}(X) \pmod{p}$. Then
\begin{align*}
N_{1} \left( p, \left[ \frac{p}{4} \right] \right) &= \frac{1}{4} \left\{ 2+ \left( 1 - \left( \frac{-1}{p}\right) \right) \left( 4 + \left( \frac{-2}{p}\right) \right) \right\} h(\sqrt{-p}) - \varepsilon ,\\
N_{1} \left( p, \left[ \frac{p}{3} \right] \right) &= \delta \, (2L(p)-1),
\end{align*}
where the number $L(p)$ appears in Theorem \ref{thm:NofLFofssp}.
\end{thm}
We note that the polynomial $W_{m}(X)$ and the Legendre polynomial
\[ P_{n}(x) := \frac{1}{2^{n} n!} \frac{d^{n}}{dx^{n}} (x^{2}-1)^{n}  = \left(\frac{x-1}{2}\right)^{n} \sum_{r=0}^{n} \binom{n}{r}^2 \left( \frac{x+1}{x-1}\right)^{r} \]
satisfy the following relation:
\begin{equation}
W_{m}(X) = (1-X)^{m} P_{m} \left( \frac{1+X}{1-X} \right) . \label{eq:wandl}
\end{equation}
Moreover, Morton determined the number of certain quadratic factors of the Legendre polynomials.
\begin{thm}[Morton \cite{morton2010legendre, morton2011cubic}]\label{thm:mortonbqf}
Let $p \geq 5$ be a prime and $B(p,m)$ be the number of irreducible quadratic factors of the form $X^2 +C$ of the Legendre polynomial $P_{m}(X) \pmod{p}$. Then
\begin{align*}
B \left( p, \left[ \frac{p}{4} \right] \right) = \frac{1}{4} \left\{ h(\sqrt{-2 p}) - 2 (\varepsilon + \nu ) \right\} , 
\quad  B \left( p, \left[ \frac{p}{3} \right] \right) = \frac{1}{4} \left\{ a_{p} h(\sqrt{-3 p}) -4  \delta \right\} ,
\end{align*}
where
\[ a_{p}=\frac{1}{2} \left\{ 2+ \left( 1 + \left( \frac{-1}{p}\right) \right) \left( 2 + \left( \frac{-2}{p}\right) \right) \right\}. \]
\end{thm}


\subsection{Supersingular polynomials for $\Gamma _{0}(N) \; (N=2,3) $}
Tsutsumi introduced the supersingular polynomials for congruence subgroups of low levels in \cite{tsutsumi2007atkin} and obtained hypergeometric representations for them.
There exist algebraic relations between the elliptic modular invariant $j(\tau )$ and modular functions $j_{N}(\tau )$ for $\Gamma_{0}(N)\; (N=2,3)$:
\begin{equation}
\begin{split}
( j_{2}(\tau ) +192 )^{3} - j(\tau )\, (j_{2}(\tau ) -64)^{2} =0, \\
j_{3}(\tau )\, ( j_{3}(\tau ) +216 )^{3} - j(\tau )\, (j_{3}(\tau ) -27)^{3} =0. 
\end{split} \label{eq:algreljjn}
\end{equation}
We prepare the set
\begin{align*}
S_{N} := \{ j_{N} \in \overline{\mathbb{F}}_{p} \,|\,  \text{the $j$-invariant determined by (\ref{eq:algreljjn}) is supersingular}  \}.
\end{align*}
We then define
\[ ss_{p}^{(N)}(X) = \prod_{ j_{N} \in S_{N}} (X-j_{N}) \in \overline{\mathbb{F}}_{p}[X] \]
for the prime $p \ge 5$. Note that we ignore the duplication of elements of the set $S_{N}$.
Because the set $S_{N}$ is stable under the conjugation over $\mathbb{F}_{p}$, we have $ss^{(N)}_{p}(X) \in \mathbb{F}_{p}[X]$.
The following proposition are stated in \cite[Proof of Theorem 4]{tsutsumi2007atkin} as a congruence relations of $ss^{(N)}_{p}(X)$ and 
certain polynomial like $U^{\varepsilon}_{n}(X)$ appearing in the proof of  Theorem~\ref{thm:DEU}.
\begin{prop}[Tsutsumi]
\begin{enumerate}
\item Let $p \geq 5$ be a prime and $m=[p/4]$. Then
\begin{equation}
ss^{(2)}_{p} (X) = X^{m+\varepsilon} {}_{2}F_{1}\left( -m ,\frac{3}{4}-\frac{\varepsilon}{2};1;\frac{64}{X} \right) \pmod{p} . \label{eq:hgrepof2ssp}
\end{equation}
\item Let $p \geq 5$ be a prime and $m=[p/3]$. Then
\begin{equation}
ss^{(3)}_{p} (X) = X^{m+\delta} {}_{2}F_{1}\left( -m ,\frac{2}{3}-\frac{\delta}{3};1;\frac{27}{X} \right) \pmod{p} .\label{eq:hgrepof3ssp}
\end{equation}
\end{enumerate}
\end{prop}
The degrees of these polynomials are given as follows.
\[ \deg ss^{(2)}_{p} (X) =  \frac{p-1}{4} +\frac{1}{2} \varepsilon ,\quad \deg ss^{(3)}_{p} (X) =  \frac{p-1}{3} +\frac{2}{3} \delta . \]


\begin{prop}
If $p\geq 5$, the number $L^{(N)}(p)\, (N=2,3)$ of characteristic $p$ supersingular $j_{N}$-invariants that lie in $\mathbb{F}_{p}$ is 
\[   L^{(2)}(p) = N_{1} \left( p, \left[ \frac{p}{4} \right] \right) + \varepsilon ,
\quad   L^{(3)}(p) =  N_{1} \left( p, \left[ \frac{p}{3} \right] \right)  + \delta . \]
Therefore, by Theorems \ref{thm:NofLFofssp} and \ref{thm:numoflinw}, the numbers $L^{(N)}(p)\, (N=2,3)$ are constant multiples of the class number $h(\sqrt{-p})$.
\end{prop}
\begin{proof}
We only prove the first equality, the other cases being similar.
For $m=[p/4]$,
\[ W_{m}(X) = {}_{2}F_{1} (-m, -m;1;X) = (1-X)^{m} {}_{2}F_{1} \left( -m, m+1;1; \frac{X}{X-1} \right) \]
and since $\frac{3}{4}-\frac{\varepsilon }{2} \equiv m+1  \pmod{p}$,
\begin{equation}
\begin{split}
ss^{(2)}_{p} (Y) &= Y^{m+\varepsilon} {}_{2}F_{1}\left( -m ,\frac{3}{4}-\frac{\varepsilon}{2};1;\frac{64}{Y} \right) \\
                              &\equiv Y^{\varepsilon} (Y-64)^{m} W_{m}\left( \frac{64}{64-Y} \right) \pmod{p}. 
\end{split}\label{eq:ssp2andw}
\end{equation}
Because the number of linear factors is unchanged under the linear fractional transformation of variable, we have $ L^{(2)}(p) = \varepsilon + N_{1} (p, [p/4])$ by Theorem~\ref{thm:numoflinw}.
\end{proof}
We can regard the polynomial $ss_{p}^{(2)}(X) \pmod{p} $ as the polynomial part of 
 \[ X^{[p/4]+\varepsilon } {}_{2}F_{1}\left( \frac{1}{4} ,\frac{3}{4};1;\frac{64}{X} \right), \]
and hence easy calculation of binomial coefficients gives the following explicit formula of $ss_{p}^{(2)}(X)$.
The case of $ss_{p}^{(3)}(X)$ can be shown similarly.


\begin{thm} For a prime $p\geq 5$,
\begin{align*}
ss_{p}^{(2)}(X) &= \sum_{n=0}^{(p-1 + 2\varepsilon )/4 } \binom{2 n}{n} \binom{4 n}{2 n} X^{(p-1 + 2\varepsilon )/4 -n} \pmod{p}, \\
ss_{p}^{(3)}(X) &= \sum_{n=0}^{(p-1 +2\delta )/3} \binom{2 n}{n} \binom{3 n}{n} X^{(p-1 +2\delta )/3 -n} \pmod{p}.
\end{align*}
\end{thm}
Similarly, we obtain the following theorem concerning the square of supersingular polynomials by applying Clausen's formula
\begin{equation}
{}_{2}F_{1} ( \alpha ,\beta ;\alpha +\beta +1/2 ;x)^{2} = {}_{3}F_{2} (2\alpha ,2\beta ,\alpha +\beta ; 2\alpha +2\beta , \alpha +\beta +1/2 ;x). \label{eq:clausensformula}
\end{equation}
\begin{thm}\label{thm:sqofssp123}
For a prime $p\geq 5$,
\begin{align*}
ss_{p}(X)^{2} &= (X -1728)^{\varepsilon } \sum_{n=0}^{(p-1+ 8\delta )/6} \binom{2 n}{n} \binom{3 n}{n} \binom{6 n}{3 n} X^{(p-1+ 8\delta )/6 -n}
\pmod{p}, \\
ss_{p}^{(2*)}(X)^{2} &= (X-256)^{\nu } \sum_{n=0}^{(p-1+ 6\varepsilon )/4 } \binom{2 n}{n}^{2} \binom{4 n}{2 n} X^{(p-1+ 6\varepsilon )/4 -n}
\pmod{p}, \\
ss_{p}^{(3*)}(X)^{2} &= X^{\delta } (X-108)^{\delta } \sum_{n=0}^{(p-1 +2\delta )/3} \binom{2 n}{n}^{2} \binom{3 n}{n} X^{(p-1 +2\delta )/3 -n}
\pmod{p}.
\end{align*}
\end{thm}

Here, we would like to point out the similarity between Theorem  \ref{thm:sqofssp123} and the expansions of certain Eisenstein series in terms of $1/j(\tau )$ and
similar local parameters described below. By Theorem \ref{thm:DEU}, we have
\begin{align*}
&{} ss_{p}(X) = X^{m+\delta}(X-1728)^{\varepsilon} {}_{2}F_{1}\left( -m ,\frac{5}{12}-\frac{2\delta-3\varepsilon}{6};1;\frac{1728}{X} \right) \pmod{p} \\
        &\equiv \begin{cases}
              X^{m+\delta} {}_{2}F_{1}\left( \dfrac{1}{12} ,\dfrac{5}{12} ;1;\dfrac{1728}{X} \right) \pmod{p} & \text{if $p \equiv 1 \pmod{4}$,}\\
              X^{m+\delta}(X-1728) {}_{2}F_{1}\left( \dfrac{7}{12} ,\dfrac{11}{12};1;\dfrac{1728}{X} \right) \pmod{p} & \text{if $p \equiv 3 \pmod{4}$.}
              \end{cases}
\end{align*}
We note the following hypergeometric transformation
\begin{equation}
 {}_{2}F_{1}\left( \alpha , \beta ; \gamma  ; z \right) = (1-z)^{\gamma -\alpha -\beta } {}_{2}F_{1}\left( \gamma -\alpha ,\gamma -\beta ; \gamma ;z  \right) \label{eq:kummer}
\end{equation}
and hence have
\[ {}_{2}F_{1}\left( \dfrac{1}{12} ,\dfrac{5}{12} ;1;\dfrac{1728}{X} \right) = \left( 1-\frac{1728}{X} \right)^{1/2} {}_{2}F_{1}\left( \dfrac{7}{12} ,\dfrac{11}{12} ;1;\dfrac{1728}{X} \right). \]
It is well known that the Eisenstein series of weight 4 on $SL_{2}(\mathbb{Z})$ has the following hypergeometric representation for sufficiently large $\Im (\tau)$:
\[ E_{4}(\tau ) = {}_{2}F_{1}\left( \dfrac{1}{12} ,\dfrac{5}{12} ;1;\dfrac{1728}{j(\tau )} \right)^{4} \in M_{4}(SL_{2}(\mathbb{Z})) . \]

In \cite{sakai2011atkin}, it is shown that the Eisenstein series of weight 4 on the Fricke group $\Gamma_{0}^{*}(N) \;(N=2,3)$ has the following hypergeometric series expression:
\begin{align*}
(2 E_{2}(2\tau )- E_{2}(\tau ))^{2}  &= {}_{2}F_{1}\left( \frac{1}{8} ,\frac{3}{8};1;\frac{256}{j_{2}^{*}(\tau )} \right)^{4} \in M_{4}(\Gamma_{0}^{*}(2)) ,\\
\left( \frac{3 E_{2}(3\tau )- E_{2}(\tau )}{2} \right)^{2} &= {}_{2}F_{1}\left( \frac{1}{6} ,\frac{1}{3};1;\frac{108}{j_{3}^{*}(\tau )} \right)^{4} \in M_{4}(\Gamma_{0}^{*}(3)).
\end{align*}
The calculation of  the binomial coefficients using Clausen's formula (\ref{eq:clausensformula}) gives the following:
\begin{align*}
E_{4}(\tau )^{1/2} &= \sum_{n=0}^{\infty} \binom{2 n}{n} \binom{3 n}{n} \binom{6 n}{3 n} j(\tau )^{-n} , \\
2 E_{2}(2\tau )- E_{2}(\tau ) &= \sum_{n=0}^{\infty } \binom{2 n}{n}^{2} \binom{4 n}{2 n} j_{2}^{*}(\tau )^{-n} , \\
\frac{3 E_{2}(3\tau )- E_{2}(\tau )}{2} &= \sum_{n=0}^{\infty } \binom{2 n}{n}^{2} \binom{3 n}{n} j_{3}^{*}(\tau )^{-n} .
\end{align*}
These expressions are very similar to the polynomial of Theorem \ref{thm:sqofssp123}.
Based on these similarities, we propose conjectural expressions of squares of the supersingular polynomials $ss_{p}^{(5*)}(X)$ and  $ss_{p}^{(7*)}(X)$ in Section \ref{Conjectures and Observations}.


\section{Proof of main theorems}\label{Proof of main theorems}
In order to prove Theorems \ref{thm:NofLF} and \ref{thm:relofspgp}, we start with the next proposition on the algebraic relation between the polynomial $ss^{(N*)}_{p}(X)$ and $ss^{(N)}_{p}(X)$ for $N=2,3$.
\begin{prop}\label{prop:algrelofssp}
\begin{enumerate}
\item Let $p \geq 5$ be a prime and $m=[p/8]$. Then we have
\[ (X-64)^{m+\varepsilon +\nu} ss^{(2*)}_{p} \left( \frac{X^2}{X-64} \right) =  X^{\varepsilon } (X-128)^{\nu } ss^{(2)}_{p} (X) 
\pmod{p}. \]
\item Let $p \geq 5$ be a prime and $m=[p/6]$. Then we have
\[ (X-27)^{m+2\delta } ss^{(3*)}_{p} \left( \frac{X^2}{X-27} \right) = X^{\delta }(X-54)^{\delta } ss^{(3)}_{p} (X) \pmod{p}  . \]
\end{enumerate}
\end{prop}

\begin{proof}
We only prove the case of level 2, the other case being similar.
Now $m=[p/8]$ and so $p-1= 4(2m+\nu) + 2\varepsilon $, the polynomial (\ref{eq:hgrepof2ssp}) is
\begin{equation}
ss^{(2)}_{p} (X) = X^{2m+\varepsilon +\nu} {}_{2}F_{1}\left( -2m -\nu ,\frac{3}{4}-\frac{\varepsilon }{2};1;\frac{64}{X} \right) \pmod{p}. \label{eq:propss2p}
\end{equation}
For $\alpha = -m - \nu /2, \beta = 3/8 - \varepsilon /4$, we have $\alpha +\beta +1/2 \equiv 1 \pmod{p} $.
Based on this, we apply  the formula (Kummer's relation)
\[ {}_{2}F_{1}\left( 2\alpha , 2\beta ; \alpha +\beta +\frac{1}{2} ; z \right) = {}_{2}F_{1}\left( \alpha , \beta ; \alpha +\beta +\frac{1}{2} ; 4z(1-z) \right) \]
to the right-hand side of (\ref{eq:propss2p}) and get
\begin{equation}
ss^{(2)}_{p} (X) = X^{2m+\varepsilon +\nu} {}_{2}F_{1}\left( -m -\frac{\nu }{2} , \frac{3}{8}-\frac{\varepsilon }{4};1;  \frac{256(X-64)}{X^2} \right) 
\pmod{p} .  \label{eq:propss2p1}
\end{equation}
If $\nu =0$, there is nothing to do:
\begin{equation}
ss^{(2)}_{p} (X) = X^{2m+\varepsilon} {}_{2}F_{1}\left( -m , \frac{3}{8}-\frac{\varepsilon }{4} ; 1 ; \frac{256(X-64)}{X^2} \right)  \pmod{p}. \notag
\end{equation}
If $\nu =1$, we apply (\ref{eq:kummer}) to the left-hand side of (\ref{eq:propss2p1}). Then the first two parameters of the hypergeometric series of (\ref{eq:propss2p1}) reduce modulo $p$ to
\[ 1-\alpha \equiv m+\frac{3}{2} \equiv \frac{3}{8}-\frac{\varepsilon -2}{4} , \quad  1-\beta \equiv \frac{5}{8}+\frac{\varepsilon }{4} \equiv -m \pmod{p}.  \]
Since  the exponent is $1 -\alpha -\beta \equiv 1/2 \pmod{p}$,  
\[ \left( 1- \frac{256(X-64)}{X^2} \right)^{1/2} = \frac{X-128}{X}. \]
Therefore when $\nu=1$, we have
\begin{align*}
ss^{(2)}_{p} (X) = X^{2m+\varepsilon}(X-128)\; {}_{2}F_{1}\left( -m , \frac{5}{8}+\frac{\varepsilon }{4} ; 1 ; \frac{256(X-64)}{X^2} \right)  \pmod{p} .
\end{align*}
Summarizing these cases, we finally obtain
\begin{equation}
ss^{(2)}_{p} (X) = X^{2m+\varepsilon}(X-128)^{\nu }\; {}_{2}F_{1}\left( -m , \frac{3}{8}-\frac{\varepsilon -2}{4} ; 1 ; \frac{256(X-64)}{X^2} \right) \pmod{p} .
\end{equation}
By multiplying both sides of the above equation by $X^{\varepsilon}(X-128)^{\nu }$, we get the assertion:
\begin{align*}
&{} X^{\varepsilon}(X-128)^{\nu } ss^{(2)}_{p} (X) \\
&= (X-64)^{m+\varepsilon +\nu } \left( \frac{X^2}{X-64} \right)^{m+\varepsilon } \left( \frac{X^2}{X-64} -256 \right)^{\nu } \\
&{} \quad \times {}_{2}F_{1}\left( -m , \frac{3}{8}-\frac{\varepsilon -2\nu }{4} ; 1 ; \frac{256(X-64)}{X^2} \right)  \\
&{} =(X-64)^{m+\varepsilon +\nu } ss^{(2*)}_{p} \left( \frac{X^2}{X-64} \right) \pmod{p}.
\end{align*}
\end{proof}


\begin{lem}\label{lem:therootsofW4}
The roots of $W_{(p-1-2\varepsilon)/4}(X)$ over $\mathbb{F}_{p}$ are distinct and all lie in $\mathbb{F}_{p^2}$.
\end{lem}
\begin{proof}
Firstly, we prove that if $\lambda$ is a root of $W_{(p-1)/2}(X) $, then $1-\lambda$ is also a root of it.
The polynomial $W_{m}(X) = {}_{2}F_{1} (-m, -m;1;X)$ can be regarded as a polynomial solution of certain hypergeometric differential equation at $X=0$ (exponent zero), so this polynomial coincide with the polynomial solution of the same equation at $X=1$ (exponent zero) up to a constant multiple.
Since $\sum_{r=0}^{m}\binom{m}{r}^2 = \binom{2m}{m}$, we have
\[ W_{m}(X) = {}_{2}F_{1} (-m, -m;1;X) = \binom{2m}{m} {}_{2}F_{1} (-m, -m;-2m; 1-X). \]
It is easy to see that the congruences $\binom{p-1}{(p-1)/2} \equiv (-1)^{(p-1)/2} \not\equiv 0 \pmod{p}$ and $-2m\equiv 1 \pmod{p}$ hold when $m=(p-1)/2$, we have
\[ W_{(p-1)/2}(X) \equiv (-1)^{(p-1)/2}\; W_{(p-1)/2} (1-X) \pmod{p} \]
and the first assertion follows.

Secondly, by \cite[eq. (1.2)]{brillhart2004class}, we have
\begin{equation}
W_{(p-1)/2}(X) \equiv (1-2X)^{\varepsilon } W_{(p-1-2\varepsilon )/4}(4X(1-X)) \pmod{p}. \label{eq:algrelofw2w4}
\end{equation}
If $p\equiv1 \pmod{4}$, i.e. $\varepsilon =0$, the set of roots of  $W_{(p-1)/2}(X) $ is 
\[ \{ \lambda_{1}, \dots , \lambda_{(p-1)/4}, 1-\lambda_{1}, \dots , 1-\lambda_{(p-1)/4}\}. \]
Because $X=\lambda$ and $X=1-\lambda$ give the same value $4X(1-X)$, the set of roots of  $W_{(p-1)/4}(X) $ is
\begin{equation}
\{ 4\lambda_{1}(1-\lambda_{1}), \dots ,4 \lambda_{(p-1)/4}(1-\lambda_{(p-1)/4})\}. \label{setofrootsw4}
\end{equation}
By \cite[Appendix, Proposition 1]{brillhart2004class}, the roots of $W_{(p-1)/2}(X) $ over $\mathbb{F}_{p}$ are distinct and all lie in $\mathbb{F}_{p^2}$.
Combining this fact and the first assertion, we see that if $\alpha$ and $\beta$ with $\alpha \not\equiv \beta \pmod{p}$ are the roots of $W_{(p-1)/2}(X) $, then $\alpha \not\equiv 1-\beta \pmod{p}$ holds and so
$\alpha (1-\alpha) -\beta (1-\beta) = (\alpha -\beta)(\alpha +\beta -1) \not \equiv 0 \pmod{p}$.
Therefore the elements of the set (\ref{setofrootsw4})  are distinct and all lie in $\mathbb{F}_{p^2}$.

If $p\equiv3 \pmod{4}$,  i.e. $\varepsilon =1$, the set of roots of  $W_{(p-1)/2}(X) $ is
\[ \{ \lambda_{1}, \dots , \lambda_{(p-3)/4},1/2, 1-\lambda_{1}, \dots , 1-\lambda_{(p-3)/4}\}. \]
It can be seen that the set of roots of  $W_{(p-3)/4}(X) $ is
\[ \{ 4\lambda_{1}(1-\lambda_{1}), \dots ,4 \lambda_{(p-3)/4}(1-\lambda_{(p-3)/4})\} \]
as in the above case and these roots are distinct and all lie in $\mathbb{F}_{p^2}$.
\end{proof}


\begin{proof}[Proof of Theorem \ref{thm:NofLF}]
First, we prove the explicit formula of $L^{(2*)}(p)$.
By the algebraic relation (\ref{eq:wandl}) and Lemma \ref{lem:therootsofW4}, the Legendre polynomial $P_{2m+\nu }(X)$ factors into distinct linear and quadratic polynomials over $\mathbb{F}_{p}$ for the prime $p= 8m + 4\nu + 2\varepsilon +1 = 4(2m +\nu) + 2\varepsilon +1 \geq 5$.
From the definition, one can easily see that the parity of the Legendre polynomial is $P_{n}(-x) = (-1)^{n} P_{n}(x)$, and hence $X^{-\nu} P_{2m+\nu }(X)$ is an even polynomial.
Consequently, we can rewrite $P_{2m+\nu }(X)$ as follows:
\begin{equation}
\begin{split}
P_{2m+\nu }(X) &= e_{p}\, X^{\nu } \prod_{r=1}^{A} (X^2 - \ell_{r} ) \prod_{s=1}^{B} (X^2 - \alpha_{s} ) \\
&{} \quad \times \prod_{t} (X^2 + \beta_{t} X +\gamma_{t} ) (X^2 - \beta_{t} X +\gamma_{t} ) \pmod{p} ,
\end{split}\label{eq:factorizationofleg}
\end{equation}
where
\begin{align*}
e_{p} &= \frac{1}{2^{2m+\nu}}\binom{2(2m+\nu)}{2m+\nu} ,\;  \left( \frac{\ell_{r}}{p}\right) =+1 ,\; \left( \frac{\alpha _{s}}{p}\right) = -1 , \\
\beta_{t} &\not\equiv 0 \pmod{p}, \; \left( \frac{\beta _{t}^2 -4 \gamma _{t}}{p}\right) = -1.  
\end{align*}
According to Theorem \ref{thm:mortonbqf}, we have $B = B(p,[p/4]) = \tfrac{1}{4} \{ h(\sqrt{-2 p}) - 2 (\varepsilon + \nu ) \} $.
Using the expression (\ref{eq:ssp2andw}) and the algebraic relation (\ref{eq:wandl}), we have $ss_{p}^{(2)}(Y) = Y^{2m+\varepsilon +\nu } P_{2m+\nu }( 1- 128/Y ) \pmod{p}$.
Substituting $X=1-128/Y$ into (\ref{eq:factorizationofleg}), we have
\begin{equation}
\begin{split}
ss_{p}^{(2)}(Y) &= e_{p}\, Y^{\varepsilon } (Y-128)^{\nu } \prod_{r=1}^{A} ((Y-128)^2 - \ell_{r} Y^2 )  \prod_{s=1}^{B} ((Y-128)^2 - \alpha _{s} Y^2 )  \\
                            &{} \quad \times \prod_{t} ((Y-128)^2 + \beta_{t} Y(Y-128) +\gamma_{t}Y^2 ) \\
                            &{} \quad \times \prod_{t}((Y-128)^2 - \beta_{t} Y(Y-128) +\gamma_{t}Y^2 ) \pmod{p}. 
\end{split}\label{eq:factorizationofss2p}
\end{equation} 
By this expression, the number of linear factors of $ss_{p}^{(2)}(Y)$ is given by $L^{(2)}(p) = \varepsilon +\nu  +2A$, and hence $A=\frac{1}{2}(L^{(2)}(p) -\varepsilon -\nu )$.

We calculate linear factors and factors quadratic in $Y^{2}/(Y-64)$ in the above product (\ref{eq:factorizationofss2p}) separately.
The linear factor is
\[ (Y-128)^2 - \ell_{r} Y^2 = (Y-64) \left\{ (1-\ell_{r}) \frac{Y^2}{Y-64}  -256 \right\}, \]
the quadratic factor is
\begin{align*}
&{} ((Y-128)^2 + \beta_{t} Y(Y-128) +\gamma_{t}Y^2 ) ((Y-128)^2 - \beta_{t} Y(Y-128) +\gamma_{t}Y^2 ) \\
&= (Y-64)^2 \left\{ (1+\beta _{t} +\gamma _{t}) (1-\beta _{t} +\gamma _{t}) x^2 -256 (2-\beta_{t}^2 +2\gamma _{t}) x +256^2 \right\} \\
&=: (Y-64)^2 f_{t}(x),
\end{align*}
where $x = Y^{2}/(Y-64)$. The discriminant of  the quadratic polynomial $f_{t}(x)$ is equal to
\[ 256^2 (2-\beta_{t}^2 +2\gamma _{t})^2 -4 \cdot 256^2  (1+\beta _{t} +\gamma _{t}) (1-\beta _{t} +\gamma _{t}) = 256^2 \beta_{t}^2 (\beta _{t}^2 -4 \gamma _{t}),  \]
and by the assumption of $\beta_{t}$ and $\gamma_{t}$, we have
\[ \left( \frac{256^2 \beta_{t}^2 (\beta _{t}^2 -4 \gamma _{t})}{p} \right) = \left( \frac{\beta _{t}^2 -4 \gamma _{t}}{p} \right) =-1.  \]
Therefore the quadratic polynomial $f_{t}(Y^{2}/(Y-64))$ is irreducible. For simplicity, we put
\[ ss_{p}^{(2)}(Y) = e_{p}\, Y^{\varepsilon } (Y-128)^{\nu } (Y-64)^{m} \prod_{r=1}^{A} ( \text{lin.} )  \prod_{s=1}^{B} (\text{lin.}) \prod_{t} ( \text{irred. quad.} ),  \]
where ``lin.'' and `` irred. quad.'' mean a linear factor and a irreducible quadratic factor with $Y^{2}/(Y-64)$ as a variable, respectively.
We substitute this expression for Proposition \ref{prop:algrelofssp}:
\begin{align*}
&{} ss^{(2*)}_{p} \left( \frac{Y^2}{Y-64} \right) =  \frac{ Y^{\varepsilon } (Y-128)^{\nu}}{(Y-64)^{m+\varepsilon +\nu }} ss^{(2)}_{p} (Y) \\
&=  e_{p}\, \frac{ Y^{2\varepsilon } (Y-128)^{2\nu}}{(Y-64)^{\varepsilon +\nu }} \prod_{r=1}^{A} ( \text{lin.} )  \prod_{s=1}^{B} (\text{lin.}) \prod_{t} ( \text{irred. quad.} ) \\
&= e_{p}\,  \left( \frac{Y^2}{Y-64} \right)^{\varepsilon } \left( \frac{Y^2}{Y-64} -256 \right)^{\nu }  \prod_{r=1}^{A} ( \text{lin.} )  \prod_{s=1}^{B} (\text{lin.}) \prod_{t} ( \text{irred. quad.} )  \pmod{p}.
\end{align*}
Therefore the number of linear factors $L^{(2*)}(p)$ of $ss_{p}^{(2*)}(X)$ is given by
\begin{align*}
L^{(2*)}(p) &= \varepsilon +\nu +A +B \\
&= \varepsilon +\nu +\frac{1}{2} \left( L^{(2)}(p) -\varepsilon -\nu \right) + \frac{1}{4} \left\{ h(\sqrt{-2 p}) - 2 (\varepsilon + \nu ) \right\} \\
&= \frac{1}{2} L^{(2)}(p) + \frac{1}{4} h(\sqrt{-2 p}).
\end{align*}
By \cite[Proposition 4]{brillhart2004class}, the polynomial $W_{(p-1-\delta )/3}(1-x/27) \pmod{p}$ factors into linear and quadratic polynomials.
Therefore, in the same way as in the above case, the polynomial 
\begin{align*}
ss_{p}^{(3)}(Y) &= Y^{\delta} (Y-27)^{2m+\delta} W_{2m+\delta}\left( \frac{27}{27-Y} \right)  \\
&\equiv Y^{2m+2\delta } P_{2m+\delta } \left( 1- \frac{54}{Y} \right) \pmod{p}
\end{align*}
factors into linear and quadratic polynomials and we finally obtain
\begin{align*}
L^{(3*)}(p) &= 2\delta +\frac{1}{2} \left( L^{(3)}(p) -2\delta \right) + \frac{1}{4} \left\{ a_{p} h(\sqrt{-3 p}) -4  \delta \right\} \\
&= \frac{1}{2} L^{(3)}(p) + \frac{1}{4}  a_{p} h(\sqrt{-3 p}).
\end{align*}
\end{proof}


\begin{proof}[Proof of Theorem \ref{thm:relofspgp}]
Let $D$ be a fundamental discriminant of a quadratic field $K$ and $\chi_{D} $ be the mod $|D|$ primitive Dirichlet character. 
For $D<0$, by the Dirichlet class number formula, the class number $h_{K}$ of $K$ is given by
\[ h_{K} = \frac{w \sqrt{|D|}}{2\pi} L(1, \chi_{D}) , \quad L(s, \chi_{D} ) = \sum_{n=1}^{\infty} \frac{\chi_{D}(n)}{n^{s}}, \]
where $w$ is the number of fundamental units of $K$. For $p \ge 5, \, N \in \{1,2,3\} $ and $K=\mathbb{Q}(\sqrt{-Np})$, we have $w=2$ and 
\[ h(\sqrt{-Np}) < \frac{2\sqrt{\mathstrut Np}}{\pi} \log (4Np) \]
with the help of the well known estimate $ |L(1, \chi_{D})| < \log |D|$ (see \cite[\S 9]{bateman1950remarks}). Thus we have
\begin{align*}
\deg ss_{p}^{(2*)}(X) - L^{(2*)}(p) &\ge \frac{p-1}{8} - \left( \frac{1}{2} L^{(2)}(p) + \frac{1}{4} h(\sqrt{-2p}) \right) \\
&\ge \frac{p-1}{8} - \left( \frac{3}{2} h(\sqrt{-p}) + \frac{1}{4} h(\sqrt{-2p}) \right) \\
&> \frac{p-1}{8} - \frac{\sqrt{p}}{2 \pi} \left( 6 \log (4p) +\sqrt{2} \log (8p)  \right) \\
&> 0 \qquad \text{if $p \geq p_{1266} = 10321$},
\end{align*}
where $p_{n}$ means the $n$-th prime number. Using Mathematica, we checked directly that $\deg ss_{p}^{(2*)}(X) = L^{(2*)}(p) \Leftrightarrow p\mid \# \mathbb{B}$ for $p<10321$.
By a similar estimate, we have
\[ \deg ss_{p}^{(3*)}(X) - L^{(3*)}(p) > 0 \qquad \text{if $p \geq p_{3280} = 30341$} \]
and checked directly that $\deg ss_{p}^{(3*)}(X) = L^{(3*)}(p) \Leftrightarrow p\mid \# \mathbb{B}$ for $p<30341$.
\end{proof}

For the supersingular polynomial $ss_{p}^{(N)}(X)$ related to congruence subgroup $\Gamma_{0}(N)$, we can not find a remarkable correspondence of the number of linear factors and sporadic groups.
But when limiting the number of quadratic factor of $ss_{p}^{(N)}(X)$, the following holds.
\begin{thm}
For a prime number $p$,
\begin{align*}
0 \le \tfrac{1}{2} (\deg ss^{(2)}_{p}(X) - L^{(2)}(p) ) \le 1 &{}\iff p\mid \# \mathbb{B}  ,\\
0 \le \tfrac{1}{2} (\deg ss^{(2)}_{p}(X) - L^{(2)}(p) ) \le 3 &{}\iff p\mid \# \mathbb{M} ,\\
0 \le \tfrac{1}{2} (\deg ss^{(3)}_{p}(X) - L^{(3)}(p) ) \le 2 &{}\iff p\mid \# Fi^{\prime}_{24} ,\\
0 \le \tfrac{1}{2} (\deg ss^{(3)}_{p}(X) - L^{(3)}(p) ) \le 5 &{}\iff p\mid \# \mathbb{M}.
\end{align*}
\end{thm}
\begin{proof}
Use the class number estimate.
\end{proof}


\section{Conjectures and Observations }\label{Conjectures  and Observations}
Throughout this section, we assume that the number $N$ is a prime divisor of  the monster group $\mathbb{M}$:
\[ N \in  \mathfrak{S} := \{ 2,3,5,7,11,13,17,19,23,29,31,41,47,59,71 \}. \]
It is well known that the modular curve $X_{0}(N)^{+}$ for prime $N$ has genus zero if and only if $N \in \mathfrak{S}$.
The Mckay-Thompson series defined in \cite{conway1979monstrous} is a generator of the field of modular functions (it is also called ``Hauptmodul'') on certain group and has some constant term.
But when obtaining the natural hypergeometric (Heun) polynomial representation of $ss_{p}^{(N*)}(X)$ for $N=2,3,5$ and $7$, we can not ignore the difference between these constant terms.
Therefore we need to choose the appropriate constant term as follows: the Hauptmodul $j_{N}^{*}(\tau )$ on the Fricke group $\Gamma_{0}^{*}(N) \;(N=2,3,5,7)$ is defined by
\begin{align*}
j_{2}^{*}(\tau ) &= \left( \frac{\eta (\tau )}{\eta (2\tau )} \right)^{24} +128 + 4096 \left( \frac{\eta (2\tau )}{\eta (\tau )} \right)^{24} = q^{-1} + 104+4372 q + \cdots , \\
j_{3}^{*}(\tau ) &= \left( \frac{\eta (\tau )}{\eta (3\tau )} \right)^{12} +54 +729 \left( \frac{\eta (3\tau )}{\eta (\tau )} \right)^{12} = q^{-1} + 42+783 q + \cdots , \\
j_{5}^{*}(\tau ) &= \left( \frac{\eta (\tau )}{\eta (5\tau )} \right)^{6} +22 +125 \left( \frac{\eta (5\tau )}{\eta (\tau )} \right)^{6} = q^{-1} + 16+134 q+ \cdots , \\
j_{7}^{*}(\tau ) &= \left( \frac{\eta (\tau )}{\eta (7\tau )} \right)^{4} +13 +49 \left( \frac{\eta (7\tau )}{\eta (\tau )} \right)^{4} = q^{-1} + 9+51 q+ \cdots ,
\end{align*}
where $q=e^{2\pi i \tau}$ and $\eta (\tau )$ is the Dedekind eta function defined by $ \eta (\tau ) = q^{1/24} \prod_{n=1}^{\infty} (1-q^{n})$.
Because the number of linear factors or the degree of the supersingular polynomials is unchanged by the value of the constant term of the Hauptmodul, 
we chose 0 as the value of the constant term of $j_{N}^{*}(\tau )$ for $N \in \mathfrak{S}-\{2,3,5,7\}$.


\begin{prop}
Let $N \in \mathfrak{S}$. Then there exists certain monic polynomial $a_{N}(Y)$ and $ b_{N}(Y) \in \mathbb{Z}[Y]$ of degree $N$ and $N+1$ respectively, such that
\begin{align}
j(\tau ) + j(N \tau ) = a_{N}(j_{N}^{*}(\tau )), \quad  j(\tau ) j(N \tau ) = b_{N}(j_{N}^{*}(\tau )) . \label{eq:polyab}
\end{align}
\end{prop}

\begin{proof}
We prove the first result, the second result being similar.
The Fourier coefficient of $j(\tau ) = q^{-1} + \cdots , \; j_{N}^{*}(\tau ) = q^{-1} + \cdots $ is in $\mathbb{Z}$, there exists degree $N$ polynomial $a_{N}(Y) \in \mathbb{Z}[Y]$ such that
\[ F(\tau ) := j(\tau ) + j(N \tau ) -a_{N}(j_{N}^{*}(\tau )) = O(q). \]
To prove $F(\tau ) = 0$, we check the modularity of $F(\tau )$ under the slash action. For $\alpha = \bigl( \begin{smallmatrix} N & 0 \\ 0 & 1 \end{smallmatrix} \bigr)$ and $\gamma = \bigl( \begin{smallmatrix} a & b \\ c & d \end{smallmatrix} \bigr) \in \Gamma_{0}(N) $, since $\alpha \gamma \alpha ^{-1} \in SL_{2}(\mathbb{Z})$, we have
\begin{align*}
&{} j(\tau ) |_{0}\, \gamma = j(\tau ) ,\\
&{} j(N \tau ) |_{0}\, \gamma = j(\tau ) |_{0}\, \alpha \gamma =  ( j(\tau ) |_{0}\, \alpha \gamma \alpha ^{-1} )|_{0}\, \alpha =j(\tau ) |_{0}\, \alpha = j(N \tau ) .
\end{align*}
For $\gamma \in \Gamma_{0}^{*}(N) = \langle \Gamma_{0}(N) ,\, w_{N} \rangle ,\; w_{N} := \bigl( \begin{smallmatrix} 0 & -1/\sqrt{N} \\ \sqrt{N} & 0 \end{smallmatrix} \bigr) $,  we have
\begin{align*}
&{} j(\tau ) |_{0}\, w_{N} = j \left( \frac{-1/\sqrt{N}}{\sqrt{N}\tau } \right) =j\left( \frac{-1}{N \tau } \right) =j(N \tau ) , \\
&{} j(N\tau ) |_{0}\, w_{N} = j \left( \frac{-1}{\tau } \right) =j(\tau ) .  
\end{align*}
Therefore, $(j(\tau ) + j(N \tau ) )|_{0}\, \gamma = j(\tau )+ j(N \tau )$ and by definition $j_{N}^{*}(\tau ) |_{0} \gamma = j_{N}^{*}(\tau )$ for $\gamma \in \Gamma_{0}^{*}(N)$.
It means $F(\tau ) |_{0} \gamma = F(\tau )$ for $\gamma \in \Gamma_{0}^{*}(N)$. The function $F(\tau)$ is meromorphic in the upper half plane, and hence $F(\tau )$ is constant and $F(\tau ) = F(i \infty) =0$. 
\end{proof}


We define the polynomial $R_{N}(X,Y) \in \mathbb{Z}[X,Y]$ in two variables by using the polynomials $a_{N}(Y)$ and $b_{N}(Y)$ in (\ref{eq:polyab}) as follows:
\begin{equation}
R_{N}(X,Y) := X^2 -a_{N}(Y) X +b_{N}(Y). \label{eq:RN}
\end{equation}
The equations $R_{N}(j(\tau ),\, j_{N}^{*}(\tau )) = R_{N}(j(N\tau ),\, j_{N}^{*}(\tau )) =0 $  hold trivially by the definition of $R_{N}(X,Y)$.
We note that the polynomial $a_{N}(Y)$ and $b_{N}(Y)$ depend on the constant term of $j_{N}^{*}(\tau )$.
\begin{ex}
\begin{align*}
R_{2}(X,Y) &= X^2-X(Y^2-207 Y+3456)+(Y+144)^3 ,\\
R_{3}(X,Y) &= X^2-X Y(Y^2-126 Y+2944)+Y(Y+192)^3 ,\\
R_{5}(X,Y) &= X^2- X(Y^5-80 Y^4+1890 Y^3-12600 Y^2+7776 Y+3456) \\
                   &{} \quad +(Y^2+216 Y+144)^3 ,\\
R_{7}(X,Y) &= X^2- X Y(Y^2-21 Y+8) (Y^4-42 Y^3+454 Y^2-1008 Y-1280) \\
                   &{} \quad +Y^2 (Y^2+224 Y+448)^3.
\end{align*}
\end{ex}
Following the definition of $ss_{p}^{(N)}(X)$, we now newly define the supersingular polynomials $ss_{p}^{(N*)}(X)$ (including the cases $N=2,3$). 
For $N \in \mathfrak{S}$, we prepare the set
\begin{align*}
S_{N}^{*} := \{ j_{N}^{*} \in \overline{\mathbb{F}}_{p} \,|\,  \text{the $j$-invariant determined by $R_{N}(j,\, j_{N}^{*})=0$ is supersingular}  \}
\end{align*}
and define
\begin{equation}
ss_{p}^{(N*)}(X) = \prod_{ j_{N}^{*} \in S_{N}^{*}} (X-j_{N}^{*}) \in \overline{\mathbb{F}}_{p}[X] \label{def:sspNstar}
\end{equation}
for the prime $p \ge 5$. Note that we ignore the duplication of elements of the set $S_{N}^{*}$.
Because the set $S_{N}^{*}$ is stable under the conjugation over $\mathbb{F}_{p}$,
we have $ss^{(N*)}_{p}(X) \in \mathbb{F}_{p}[X]$.

\begin{ex}
Since $ss_{37}(X) = (X+29)(X^2 +31X +31) \pmod{37}$, the resultant of $ss_{37}(X)$ and $R_{2}(X,Y)$ with respect to the variable $X$ is congruent to 
\begin{equation}
(Y+3) (Y+25)^{2} (Y+27)^{2} (Y^2+14Y+34)^{2} \pmod{37}. \label{eq:resultant37}
\end{equation}
By ignoring the multiplicity of the roots of this, we obtain the supersingular polynomial $ ss_{37}^{(2*)}(Y) $ for $\Gamma_{0}^{*}(2)$:
\[ ss_{37}^{(2*)}(Y) =  (Y+3) (Y+25) (Y+27) (Y^2+14Y+34) \pmod{37} . \]
\end{ex}

\begin{ques}
Are all supersingular $j_{N}^{*}$-invariants contained in $\mathbb{F}_{p^2}$ for any $N \in \mathfrak{S}$?
\end{ques}

This definition of $ss_{p}^{(N*)}(X)$ for $N=2,3$ and Definition \ref{def:sakaissp} by Koike and Sakai are equivalent. 
Combining the algebraic relations (\ref{eq:algreljjn}) and 
\begin{equation}
j_{2}(\tau)^{2} -j_{2}^{*}(\tau ) (j_{2}(\tau )-64) =0, \quad j_{3}(\tau)^{2} -j_{3}^{*}(\tau ) (j_{3}(\tau )-27) =0, \label{eq:algrelj23andj23star}
\end{equation}
we rewrite the definition of the set $S_{N}^{*} \,(N=2,3)$ as
\begin{align}
S_{N}^{*} = \{ j_{N}^{*} \in \overline{\mathbb{F}}_{p} \,|\,  \text{the $j_{N}$-invariant determined by (\ref{eq:algrelj23andj23star}) is supersingular}  \}. \label{def:rewriteSN}
\end{align}
We obtain the degree of $ss_{p}^{(N*)}(X)\; (N=2,3)$ from the algebraic relation (\ref{eq:algrelj23andj23star})  by the same method as in \cite[Proposition 4]{tsutsumi2007atkin}.
For $p=8m+ 4\nu +2\varepsilon +1$, we have $d :=\deg ss_{p}^{(2*)}(X) = m + \nu +\varepsilon $.

Let $ss_{p}^{(2*)}(X) = (X- \alpha_{1}) \dots (X- \alpha_{d})$ be the factorization over $\overline{\mathbb{F}}_{p}$ of  $ss_{p}^{(2*)}(X)$.
Then we have
\begin{align*}
(Y-64)^{d} ss_{p}^{(2*)} \left( \frac{Y^2}{Y-64} \right) = \prod_{n=1}^{d} \bigl( Y^2 - \alpha_{n} (Y-64) \bigr).
\end{align*}
Because of the definitions of $ss_{p}^{(2)}(Y)$ and (\ref{def:rewriteSN}), the right-hand side of the above equation is divisible by  $ss_{p}^{(2)}(Y)$.
The double roots of $Y^2 - \alpha_{n} (Y-64)$ are 0 or 128 when $\alpha_{n}=0$ or 256 respectively.
Comparing the degrees $\deg ss_{p}^{(2*)}(X) = m + \nu +\varepsilon$ and $\deg ss_{p}^{(2)}(Y) = 2m+ \nu + \varepsilon$, and noting the vanishing condition $ss_{p}^{(2)}(0)=0 \pmod{p} \iff \varepsilon =1 $ and
\[ ss_{p}^{(2)}(128) = 128^{2m + \nu + \varepsilon} P_{2m + \nu}(0) = 0 \pmod{p} \iff \nu =1, \]
we have
\begin{align*}
(Y-64)^{d} ss_{p}^{(2*)} \left( \frac{Y^2}{Y-64} \right) = Y^{\varepsilon } (Y-128)^{\nu} ss_{p}^{(2)}(Y) .
\end{align*}
Therefore we obtain the algebraic relation of $ss_{p}^{(2*)}$ and $ss_{p}^{(2)}$ and the proof of Proposition \ref{prop:algrelofssp} tells us that
the hypergeometric expression of $ss_{p}^{(2*)}(X)$ in Definition \ref{def:sakaissp} coincides with our definition. The case of $ss_{p}^{(3*)}(X)$ is similar.


\subsection{Level 5 and 7}
For $N \in \mathfrak{S}-\{2,3\}$, the supersingular polynomial $ss_{p}^{(N*)}(X)$ can no longer be represented by hypergeometric polynomial.
But we conjecture that $ss_{p}^{(N*)}(X) \;(N=5,7)$ can be represented by Heun polynomials. The Heun local series is defined by
\[ Hl (a,w; \alpha , \beta , \gamma , \delta ;x) = \sum_{n=0}^{\infty} c_{n} x^n ,\]
where the coefficients $c_{n}$ are determined by the recursion
\[ c_{n+1} = \frac{n\{ (n-1+\gamma )(1+a) +a\delta +\varepsilon \} +w  }{(n+1)(n+\gamma ) a } c_{n} - \frac{(n-1+\alpha )(n-1+\beta )}{(n+1)(n+\gamma ) a } c_{n-1}  \] 
and the initializations $c_{-1}=0, c_{0}=1$.  (Here, $\varepsilon := \alpha + \beta +1 -\gamma -\delta $ and of course, this local symbols $ \delta, \varepsilon$ are different from the previous number (\ref{eq:nudeleps}).)
This function is a solution of the Heun equation
\[ \frac{d^{2}y}{dx^{2}} + \left( \frac{\gamma }{x} + \frac{\delta }{x-1} + \frac{\varepsilon }{x-a} \right) \frac{dy}{dx} +\frac{\alpha \beta x - w}{x(x-1)(x-a)} y =0, \]
which has four singular points $x=0,1,a,\infty$. In particular, $Hl$ is a polynomial when $\alpha$ or $\beta$ is a negative integer. 
There are 192 solutions of the Heun equation as an analogue of Kummer's 24 solutions of Gauss hypergeometric differential equation (see \cite{maier2007192}).
One solution $Hl$ has 24 equivalent expressions, for example, the following identity holds generically near $x=0$:
\begin{equation}
\begin{split}
&{} Hl (a,w; \alpha , \beta , \gamma , \delta ;x) \\
&= (1-x)^{1-\delta } (1-\tfrac{x}{a})^{1-\varepsilon } Hl ( a, w' ; -\beta +\gamma +1 , -\alpha +\gamma +1 , \gamma , 2-\delta  ; x ) , 
\end{split}\label{eq:algrelheun}
\end{equation}
where $w' = w -\gamma \{ (\delta -1) a +\varepsilon -1  \}$.

\begin{conj}
\begin{enumerate}
\item For a prime number $p \ge 7$,
\[ ss_{p}^{(5*)}(X) = X^{m_{5}} (X^2-44X-16)^{\mu_{5}} Hl_{5}(X) \pmod{p} \]
where 
\begin{align*}
&{}m_{5} = \frac{p-1}{4} +\frac{1}{4}\left( 1- \left( \frac{-1}{p}\right)  \right) -\mu_{5},\quad \mu_{5}=\frac{1}{2}\left( 1- \left( \frac{-5}{p}\right)  \right) ,\quad \phi =\frac{1+\sqrt{5}}{2}, \\
&{}Hl_{5}(X) = Hl \left(- \phi^{10} , -\frac{(22\mu_{5} +3) \phi^5}{4} ; -m_{5} , \mu_{5}+\frac{1}{2} + \frac{1}{4} \left( \frac{-1}{p}\right) ,1, \mu_{5}+\frac{1}{2} ; \frac{4\phi^5}{X} \right).
\end{align*}
\item For a prime number $p=5$ and $p \ge 11$,
\[ ss_{p}^{(7*)}(X) = X^{m_{7}} (X+1)^{\mu_{7}} (X-27)^{\mu_{7}} Hl_{7}(X) \pmod{p} \]
where 
\begin{align*}
&{}m_{7} = \frac{p-1}{3} +\frac{1}{3}\left( 1- \left( \frac{-3}{p}\right)  \right) -\mu_{7},\quad \mu_{7}=\frac{1}{2}\left( 1- \left( \frac{-7}{p}\right)  \right) , \\
&{}Hl_{7}(X) = Hl \left(- 27 , -(13\mu_{7} +2) ; -m_{7} , \mu_{7} + \frac{1}{2} +\frac{1}{6}\left( \frac{-3}{p}\right) ,1, \mu_{7}+\frac{1}{2} ; \frac{27}{X} \right) .
\end{align*}
\end{enumerate}
\end{conj}

The degree of $ss^{(N*)}_{p} (X)$ with $p>N \; (N=5,7)$ is provided by
\begin{align*}
\deg ss^{(5*)}_{p} (X) &= m_{5} +2\, \mu_{5} = \frac{p-1}{4} +\frac{1}{4}\left( 1- \left( \frac{-1}{p}\right)  \right) +\frac{1}{2}\left( 1- \left( \frac{-5}{p}\right)  \right), \\[3pt]
\deg ss^{(7*)}_{p} (X) &= m_{7} +2\, \mu_{7} = \frac{p-1}{3} +\frac{1}{3}\left( 1- \left( \frac{-3}{p}\right)  \right) +\frac{1}{2}\left( 1- \left( \frac{-7}{p}\right)  \right).
\end{align*}
We note that the case of $\left( \tfrac{-5}{p}\right) = \left( \tfrac{-7}{p}\right) =1$, i.e. $\mu_{5} = \mu_{7} =0$ was conjectured by Sakai \cite{sakai2014atkin}.
The polynomials $X^2-44X-16$ and $(X+1)(X-27)$ in the above conjecture are derived from the relation of the Heun local series (\ref{eq:algrelheun}).

Let $HN$ be the Harada-Norton group and $He$ be the Held group. The orders of these groups are given by
\begin{align*}
&{} \# HN = 273030912000000 =2^{14}\cdot 3^6\cdot 5^6\cdot 7\cdot 11\cdot 19,\\
&{} \# He = 4030387200 = 2^{10}\cdot 3^3\cdot 5^2\cdot 7^3\cdot 17 .
\end{align*}
The following conjecture is an analogue of Theorem \ref{thm:NofLFofLevel1ssp} for the sporadic groups $HN$ and $He$.
\begin{conj}\label{conj:ssp57andspgp}
For a prime number $p$,
\begin{align*}
\deg ss^{(5*)}_{p}(X) = L^{(5*)}(p) &{}\iff p\mid \# HN , \\
\deg ss^{(7*)}_{p}(X) = L^{(7*)}(p) &{}\iff p\mid \# He .
\end{align*}
\end{conj}

\begin{rem}
Harada-Norton group $HN$ and Janko group $J_{1}$ has same set of prime divisors:
\begin{align*}
&{} \# HN = 273030912000000 =2^{14}\cdot 3^6\cdot 5^6\cdot 7\cdot 11\cdot 19,\\
&{} \# J_{1} = 175560 = 2^3\cdot 3\cdot 5\cdot 7\cdot 11\cdot 19 .
\end{align*}
Therefore, we can not distinguish the sporadic groups $HN$ and $J_{1}$ by comparing sets of prime divisors.
However, the so called generalized moonshine phenomena probably suggest the $HN$ group is the right one (see \cite[\S 9]{conway1979monstrous}).
\end{rem}

Sakai gives the Heun representation of certain Eisenstein series of weight 4 on $\Gamma_{0}^{*}(5)$ and $\Gamma_{0}^{*}(7)$ in \cite{sakai2014atkin}:
\begin{align*}
\widetilde{E_{4,5}}(\tau ) &:=\left( \frac{5 E_{2}(5\tau )- E_{2}(\tau )}{4} \right)^2    = Hl \left(- \phi^{10} , -\frac{3 \phi^5}{4} ; \frac{1}{4} , \frac{3}{4} , 1 , \frac{1}{2} ; \frac{4\phi^5}{j_{5}^{*}(\tau )} \right)^4 ,\\
\widetilde{E_{4,7}}(\tau ) &:= \left( \frac{7 E_{2}(7\tau )- E_{2}(\tau )}{6} \right)^2 = Hl \left(- 27 , -2 ; \frac{1}{3} , \frac{2}{3} , 1 , \frac{1}{2} ; \frac{27}{j_{7}^{*}(\tau )} \right)^4  .
\end{align*}
On the other hand, in \cite[eq.(5.41) and Theorem 7.32]{cooper2017ramanujan}, the following equations are given:
\[ \frac{5 E_{2}(5\tau )- E_{2}(\tau )}{4} = \sum_{n=0}^{\infty} \frac{ u_{5}^{*}(n)}{ j_{5}^{*}(\tau )^{n}} , \quad  \frac{7 E_{2}(7\tau )- E_{2}(\tau )}{6} = \sum_{n=0}^{\infty} \frac{u_{7}^{*}(n) }{ j_{7}^{*}(\tau )^{n}} , \]
where
\[ u_{5}^{*}(n) = \binom{2n}{n} \left\{ \sum_{k=0}^{n} \binom{n}{k}^2 \binom{n+k}{k} \right\} ,\;  u_{7}^{*}(n) =\left\{ \sum_{k=0}^{n} \binom{n}{k}^2 \binom{2 k}{n} \binom{n+k}{k} \right\} .  \]
We note that the number $u_{5}^{*}(n)/\binom{2n}{n}$ was used to prove the irrationality of $\zeta(2)$ (not $\zeta(3)$) by Ap\'{e}ry.
The conjectural relations between the square of the supersingular polynomials and Ap\'{e}ry-like numbers for levels 5 and 7 are as follows.
\begin{conj}
\begin{enumerate}
\item For a prime number $p \ge 7$,
\[ ss^{(5*)}_{p} (X)^{2} = (X^2-44X-16)^{\mu_{5}} \sum_{n=0}^{2(m_{5} +\mu_{5})} u_{5}^{*}(n)  X^{2(m_{5} +\mu_{5})-n} \pmod{p}. \]
\item For a prime number $p=5$ and $p \ge 11$,
\[ ss^{(7*)}_{p} (X)^{2} = (X+1)^{\mu_{7}} (X-27)^{\mu_{7}} \sum_{n=0}^{2(m_{7} +\mu_{7})} u_{7}^{*}(n)  X^{2(m_{7} +\mu_{7}) -n} \pmod{p}. \]
\end{enumerate}
\end{conj}

By numerical calculation, we expect the relation between the square of supersingular polynomials and the polynomial $R_{N}(X,Y)$ of (\ref{eq:RN}). We define the polynomial $V_{p}^{(N*)}(Y)$ by
\begin{equation}
V_{p}^{(N*)}(Y)  = \text{Resultant}_{X}[ ss_{p}(X), R_{N}(X,Y) ].  \label{eq:Vp} 
\end{equation}
We defined the supersingular polynomial $ss_{p}^{(N*)}(Y)$ by ignoring the multiple roots of the right-hand side of (\ref{eq:Vp}).
The following conjecture suggests that the multiplicity is almost equal to~2 (see eq.(\ref{eq:resultant37})).
\begin{conj} For a prime $p\geq 5$ and $p \not= N$, the polynomial $V_{p}^{(N*)}$ satisfies the following:
\begin{align*}
&{} Y^{\varepsilon } (Y-256)^{\nu } V_{p}^{(2*)}(Y) \\
&= (Y+144)^{2\delta } (Y-648)^{\varepsilon } ss_{p}^{(2*)}(Y)^{2} \pmod{p}, \\
&{} Y^{\delta } (Y-108)^{\delta } V_{p}^{(3*)}(Y) \\
&= (Y+192)^{2\delta } (Y^{2} -576 Y -1728)^{\varepsilon } ss_{p}^{(3*)}(Y)^{2} \pmod{p}, \\
&{} (Y^{2} -44 Y -16)^{\mu_{5}}  V_{p}^{(5*)}(Y) \\
&= (Y^{2} + 216 Y +144)^{2\delta }(Y^{2} -540 Y -6480)^{\varepsilon } ss_{p}^{(5*)}(Y)^{2} \pmod{p}, \\
&{} (Y+1)^{\mu_{7}} (Y-27)^{\mu_{7}} V_{p}^{(7*)}(Y) \\
&= (Y^{2} + 224 Y +448)^{2\delta } \\
&{} \quad \times (Y^{4} - 528 Y^{3} -9024 Y^{2} -5120 Y -1728)^{\varepsilon } ss_{p}^{(7*)}(Y)^{2} \pmod{p}.
\end{align*}
\end{conj}

\begin{rem}
The coefficients of right hand sides of above conjecture derived from $R_{N}(0,Y)$ and $R_{N}(1728,Y)$. For example,
\begin{align*}
&{} R_{7}(0,Y) = Y^2 (Y^2+224 Y+448)^{3} , \\
&{} R_{7}(1728,Y) = (Y^{4} - 528 Y^{3} -9024 Y^{2} -5120 Y -1728)^{2} . 
\end{align*}
\end{rem}


\subsection{level $N \in \mathfrak{S}-\{2,3,5,7 \}$}
Calculations on Mathematica suggest the following conjectures.
If $N \in \mathfrak{S}$ and $p$ is prime, the Kronecker symbol $\left( \tfrac{\cdot}{Np}\right) $ simply means the product of the Legendre symbol $\left( \frac{\cdot}{N}\right) $ and $\left( \frac{\cdot}{p}\right) $.
\begin{conj}
Let $p \geq 5$ be a prime number and $N \in \mathfrak{S}- \{2\}$ and $N \not= p$. Then
\begin{align*}
L^{(N*)}(p) &= \frac{1}{2} \left( 1 +\left( \frac{-p}{N}\right) \right) L(p) \\
&{}\quad +\frac{1}{8} \left\{ 2+ \left( 1 - \left( \frac{-1}{Np}\right) \right) \left( 2 + \left( \frac{-2}{Np}\right) \right) \right\} h(\sqrt{-N p}).
\end{align*}
\end{conj}

\begin{rem}By the law of quadratic reciprocity, we have
\[ \left( \frac{-p}{N}\right) = \left( \frac{-1}{N}\right) \left( \frac{ \left( \tfrac{-1}{N} \right) N}{p}\right) . \]
\end{rem}

\begin{conj}\label{conj:degssNpstar}
Let $p \geq 5$ be a prime number and $N \in \mathfrak{S}- \{2 \}$ and $N \not= p$. Then the degree of $ss^{(N*)}_{p} (X)$ is given by
\begin{align*}
\deg ss^{(N*)}_{p} (X) &= \frac{(N+1)(p-1)}{24} + \frac{1}{8} \left( 1 + \left( \frac{-1}{N}\right)\right) \left( 1 - \left( \frac{-1}{p}\right) \right) \\
                                          &{} \quad + \frac{1}{6}\left( 1 + \left( \frac{-3}{N}\right)\right) \left( 1 - \left( \frac{-3}{p}\right) \right) +\frac{1}{2} \left( 1 - \left( \frac{-N}{p}\right) \right) \deg  ss_{N}(X).
\end{align*}
\end{conj}

The number $\deg ss_{p}(X)$ has several interpretation. For $p \geq 5$,
\begin{align*}
\deg ss_{p} (X) &= 1+\dim S_{p+1} (SL_{2}(\mathbb{Z})) \\
                           &= 1+\dim S_{2} (\Gamma _{0}(p)) \\
                           &= 1+ g(X_{0}(p)) \\
                           &= h_{D_{p}},
\end{align*}
where $g(X_{0}(p))$ is the genus of the modular curve $X_{0}(p)$ and $h_{D_{p}}$ is the class number of the definite quaternion algebra $D_{p}$ (see \cite{deuring1941die,deuring1944die,eichler1938uber,igusa1958class}).
We focus on the first line of the above equality. By the dimension formula of $S_{p+1} (\Gamma_{0}^{*}(N))$ in \cite[Lemma 2.2]{choi2013basis}, we have
\[ \deg ss^{(2*)}_{p} (X) = 1 +\dim S_{p+1} (\Gamma_{0}^{*}(2)) \]
and under the above Conjecture \ref{conj:degssNpstar},
\begin{equation}
\deg ss_{p}^{(N*)}(X) = 1 +\dim S_{p+1} (\Gamma_{0}^{*}(N)) + \frac{1}{2} \left( \left( \frac{-1}{p}\right) - \left( \frac{-N}{p}\right)\right) \deg  ss_{N}(X) 
\end{equation}
for $N \in \mathfrak{S}- \{2 \}$ (the case of $N=3$ is true).


\subsection{Other class}
In this section, we considered the supersingular polynomial for the group 
\[ \Gamma_{0}(3|3) = \left\{ \begin{pmatrix} a & b/3 \\ 3c & d \end{pmatrix} \, \middle| \,  \begin{pmatrix} a & b \\ c & d \end{pmatrix} \in SL_{2}(\mathbb{Z}) \right\}.  \]
This group is labeled ``3C'' in Table 2 of \cite{conway1979monstrous} \footnote{Groups $SL_{2}(\mathbb{Z}), \Gamma_{0}^{*}(2), \Gamma_{0}(2), \Gamma_{0}^{*}(3),$ and $\Gamma_{0}(3)$ correspond to labels 1A, 2A, 2B, 3A, and 3B, respectively.} and corresponding modular function $j_{3C}(\tau ) = q^{-1} +248 q^2 + 4124 q^5 +\cdots $ satisfies
\[ j_{3C}(\tau )^{3} = j(3\tau ) = \frac{j_{3}(\tau ) (j_{3}(\tau ) -24)^{3}}{j_{3}(\tau ) -27 }.  \]
Since algebraic relation (\ref{eq:algreljjn}), we can also find the algebraic relation of $j(\tau )$ and $j_{3C}(\tau )$:
\begin{align*}
R_{3 C}(X,Y) &= X^4 -X^3 (Y^9-2232 Y^6+1069956 Y^3-36864000) + \cdots \\
                       &{} \quad  + Y^3 (Y^3+12288000)^3 .
\end{align*}
By the definition, this polynomial satisfies $R(j(\tau ), j_{3C}(\tau ) ) = \Phi_{3} (j(\tau ), j(3\tau ))  = 0$, where $\Phi_{3}(X,Y)$ is the classical modular polynomial of level 3.
We use $R_{3 C}(X,Y)$ to define the supersingular polynomial $ ss_{p}^{(3C)}(X)$ as in (\ref{def:sspNstar}),
then the degree and the number $L^{(3C)}(p)$ of linear factors of $ ss_{p}^{(3C)}(X)$ are conjectured as below.
\begin{conj}
For a prime number $p\geq 5$,
\[  \deg ss_{p}^{(3C)}(X) = \frac{p-1}{4} +\frac{3}{2} \varepsilon , \quad L^{(3C)}(p) = \left( 2 + \left( \frac{-3}{p}\right) \right) L(p). \]
where $L(p)$ is the number of  linear factors of $ss_{p}(X)$.
\end{conj}

Let $Th$ be the Thompson group that is one of the sporadic groups. The order of this group is given by
\[ \# Th = 90745943887872000 = 2^{15}\cdot 3^{10}\cdot 5^3\cdot 7^2\cdot 13\cdot 19\cdot 31. \]
The following conjecture is an analogue of Theorem \ref{thm:NofLFofLevel1ssp} for the group $Th$.
\begin{conj}
 For a prime number $p$,
\[ \deg ss^{(3C)}_{p}(X) = L^{(3C)}(p) \iff p\mid \# Th . \]
\end{conj}
The conjectural relations between the square of the supersingular polynomial and Ap\'{e}ry-like numbers\footnote{See \cite[Sequence \textbf{B}]{zagier2009integral}.} for ``class 3C '' is as follows.
\begin{conj}
 For a prime number $p\geq 5$,
\begin{align*}
&{} ss^{(3 C)}_{p}(X)^{2} = \{ (X-12)(X^2 +12 X +144) \}^{\varepsilon } \\
&{} \times  \sum_{n=0}^{(p-1)/2} \binom{2 n}{n} \left\{ \sum_{k=0}^{[n/3]} (-3)^{n-3 k} \binom{2 k}{k} \binom{3 k}{k} \binom{n}{3 k}   \right\} X^{(p-1)/2-n} \pmod{p}.  
\end{align*}
\end{conj}


\subsection{Supersingular polynomials for small prime number $p$}
In this subsection, we consider supersingular polynomials for a fixed small prime number $p$.
From $ss_{2}(X) = X \pmod{2}$ and $ss_{3}(X) = X \pmod{3}$, the polynomial $ss_{p}^{(N*)}(Y) $ can be obtained by calculating $R_{N}(0,Y) \pmod{p}$.
Possible interpolation formulas of $\deg ss^{(N*)}_{2} (X)$ and $\deg ss^{(N*)}_{3} (X)$ are as follows.
\begin{obs}
We have $\deg ss^{(2*)}_{2} (X) =\deg ss^{(2*)}_{3} (X) =\deg ss^{(3*)}_{2} (X) =\deg ss^{(3*)}_{3} (X) =1 $. For $N \in \mathfrak{S}- \{2, 3 \}$,
\begin{align*}
&{} \deg ss^{(N*)}_{2} (X) \\
&= \frac{N+7}{12} + \frac{1}{6}\left( 1 + \left( \frac{-3}{N}\right)\right) -\frac{1}{16} \left( 1 - \left( \frac{-1}{N}\right)\right) \left( 3 - \left( \frac{-2}{N}\right) \right) \deg  ss_{N}(X), \\
&{} \deg ss^{(N*)}_{3} (X) \\
&= \frac{N+1}{12} + \frac{1}{4} \left( 1 + \left( \frac{-1}{N}\right)\right) + \frac{1}{6}\left( 1 + \left( \frac{-3}{N}\right)\right) +\frac{1}{2} \left( 1 - \left( \frac{-3}{N}\right) \right) \deg  ss_{N}(X).                       
\end{align*}
\end{obs}
Similarly, we can also interpolate the number of linear factors $L^{(N*)}(p)$ of $ss_{p}^{(N*)}(Y)$ for $p=2,3$.
\begin{obs}
We have $L^{(2*)}(2) = L^{(2*)}(3) = L^{(3*)}(2) = L^{(3*)}(3) =1 $. For $N \in \mathfrak{S}- \{2, 3 \}$,
\begin{align*}
L^{(N*)}(2) &= \frac{1}{4} \left( 1 +\left( \frac{-1}{N}\right) \right) +\frac{1}{4} \left( 1 +\left( \frac{-2}{N}\right) \right) +\frac{1}{4} h(\sqrt{-2 N}) , \\
L^{(N*)}(3) &= \frac{1}{2} \left( 1 +\left( \frac{-3}{N}\right) \right) + \frac{1}{8} \left\{ 2+ \left( 1 + \left( \frac{-1}{N}\right) \right) \left( 2 + \left( \frac{-2}{N}\right) \right) \right\} h(\sqrt{-3 N}).
\end{align*}
\end{obs}
Of course, since the set $\mathfrak{S}$ is a finite set, there are countless interpolation formula as described above.

Theorems \ref{thm:relofspgp} and Conjecture \ref{conj:ssp57andspgp} are assertions about a set of prime number $p$ having certain properties with respect to a fixed level $N$, e.g. for the case of $N=2$,
\[ \quad \; \deg ss^{(2*)}_{p}(X) = L^{(2*)}(p) \iff p\mid \# \mathbb{B}. \]
However, there is a curious ``duality" between prime $p$ and level $N$ in $ss_{p}^{(N*)}(X)$. More precisely, for a set of level $N$ and a fixed small prime number $p$, the following holds.
\begin{obs}
For $N \in \mathfrak{S}$,
\begin{align*}
\deg ss^{(N*)}_{2}(X) = L^{(N*)}(2) &{}\iff N  \mid \# \mathbb{B}, \\
\deg ss^{(N*)}_{3}(X) = L^{(N*)}(3) &{}\iff N  \mid \# Fi^{\prime}_{24} , \\
\deg ss^{(N*)}_{5}(X) = L^{(N*)}(5) &{}\iff N  \mid \# HN , \\
\deg ss^{(N*)}_{7}(X) = L^{(N*)}(7) &{}\iff N  \mid \# He .
\end{align*}
\end{obs}
It can be checked by direct calculation for fifteen levels $N$.
We note that the set of prime divisors of the group $\mathbb{B}, Fi^{\prime}_{24}, HN$ and $ He$ are truly included in the set $\mathfrak{S}$.


\section*{Acknowledgements}
The author thanks to Professor Masanobu Kaneko for valuable comments, 
and for introducing him to the study of the connection between modular forms and supersingular polynomials.
He would also like to thank Professor Hiroyuki Tsutsumi and Dr. Yuichi Sakai, who introduced him to the study of supersingular polynomials with certain levels.


\begin{thebibliography}{10}

\bibitem{aricheta2018supersingular}
V.~M. Aricheta.
\newblock Supersingular elliptic curves and moonshine.
\newblock preprint, 2018.

\bibitem{bateman1950remarks}
P.~T. Bateman, S.~Chowla, and P.~Erd\H{o}s.
\newblock Remarks on the size of ${L}(1,\chi)$.
\newblock {\em Publ. Math. Debrecen}, {\bf 1}:165--180, 1950.

\bibitem{brillhart2004class}
J.~Brillhart and P.~Morton.
\newblock Class numbers of quadratic fields, {H}asse invariants of elliptic
  curves, and the supersingular polynomial.
\newblock {\em Journal of Number Theory}, {\bf 106}(1):79--111, 2004.

\bibitem{choi2013basis}
S.~Choi and C.~H. Kim.
\newblock Basis for the space of weakly holomorphic modular forms in higher
  level cases.
\newblock {\em Journal of Number Theory}, {\bf 133}(4):1300--1311, 2013.

\bibitem{conway1979monstrous}
J.~H. Conway and S.~P. Norton.
\newblock Monstrous moonshine.
\newblock {\em Bull. London Math. Soc}, {\bf 11}(3):308--339, 1979.

\bibitem{cooper2017ramanujan}
S.~Cooper.
\newblock {\em Ramanujan's theta functions}.
\newblock Springer, 2017.

\bibitem{deuring1941die}
M.~Deuring.
\newblock Die {T}ypen der {M}ultiplikatorenringe elliptischer
  {F}unktionenk\"{o}rper.
\newblock {\em Abh. Math. Sem. Hamburg}, {\bf 14}:197--272, 1941.

\bibitem{deuring1944die}
M.~Deuring.
\newblock Die {A}nzahl der {T}ypen von {M}aximalordungen einer definiten
  {Q}uaternionenalgebra mit primer {G}rundzahl.
\newblock {\em Jahresber. Deutsch. Math. Verein}, {\bf 54}:21--41, 1944.

\bibitem{duncan2016jack}
J.~F.~R. Duncan and K.~Ono.
\newblock The {J}ack {D}aniels {P}roblem.
\newblock {\em Journal of Number Theory}, {\bf 161}:230--239, 2016.

\bibitem{eichler1938uber}
M.~Eichler.
\newblock {\"U}ber die {I}dealklassenzahl total definiter
  {Q}uaternionenalgebren.
\newblock {\em Math. Z}, {\bf 43}:102--109, 1938.

\bibitem{igusa1958class}
J.~Igusa.
\newblock Class number of a definite quaternion with prime discriminant.
\newblock {\em Proceedings of the National Academy of Sciences}, {\bf
  44}(4):312--314, 1958.

\bibitem{kaneko1998supersingular}
M.~Kaneko and D.~Zagier.
\newblock Supersingular $j$-invariants, hypergeometric series, and {A}tkin's
  orthogonal polynomials.
\newblock {\em AMS/IP Studies in Advanced Mathematics}, {\bf 7}:97--126, 1998.

\bibitem{koike2009supersingular}
M.~Koike.
\newblock On supersingular $j_{2}^{*}$-polynomials for ${\Gamma}_{0}^{*}(2)$.
\newblock unpuplished, 2009.

\bibitem{maier2007192}
R.~Maier.
\newblock The 192 solutions of the {H}eun equation.
\newblock {\em Mathematics of Computation}, {\bf 76}(258):811--843, 2007.

\bibitem{morton2010legendre}
P.~Morton.
\newblock Legendre polynomials and complex multiplication, {I}.
\newblock {\em Journal of Number Theory}, {\bf 130}(8):1718--1731, 2010.

\bibitem{morton2011cubic}
P.~Morton.
\newblock The cubic {F}ermat equation and complex multiplication on the
  {D}euring normal form.
\newblock {\em The Ramanujan Journal}, {\bf 25}(2):247--275, 2011.

\bibitem{ogg1975automorphismes}
A.~P. Ogg.
\newblock Automorphismes de courbes modulaires.
\newblock {\em S{\'e}minaire Delange-Pisot-Poitou. Th{\'e}orie des nombres},
  {\bf 16}(1):1--8, 1975.

\bibitem{pizer1978note}
A.~Pizer.
\newblock A note on a conjecture of {H}ecke.
\newblock {\em Pacific Journal of Mathematics}, {\bf 79}(2):541--548, 1978.

\bibitem{sakai2011atkin}
Y.~Sakai.
\newblock The {A}tkin orthogonal polynomials for the low-level {F}ricke groups
  and their application.
\newblock {\em International Journal of Number Theory}, {\bf 7}(06):1637--1661,
  2011.

\bibitem{sakai2014atkin}
Y.~Sakai.
\newblock The {A}tkin orthogonal polynomials for the {F}ricke groups of levels
  5 and 7.
\newblock {\em International Journal of Number Theory}, {\bf
  10}(08):2243--2255, 2014.

\bibitem{silverman2009arithmetic}
J.~H. Silverman.
\newblock {\em The arithmetic of elliptic curves}, volume 106 of {\em Graduate
  Texts in Mathematics}.
\newblock Springer New York, 2nd edition, 2009.

\bibitem{tsutsumi2007atkin}
H.~Tsutsumi.
\newblock The {A}tkin orthogonal polynomials for congruence subgroups of low
  levels.
\newblock {\em The Ramanujan Journal}, {\bf 14}(2):223--247, 2007.

\bibitem{zagier2009integral}
D.~Zagier.
\newblock Integral solutions of {A}p{\'e}ry-like recurrence equations.
\newblock {\em Groups and Symmetries: from Neolithic Scots to John McKay, CRM
  Proc. Lecture Notes}, {\bf 47}:349--366, 2009.

\end{thebibliography}

\end{document}